\documentclass[11pt,reqno]{amsart}
\usepackage{amsmath}
\usepackage{amssymb}
\usepackage{latexsym}
\usepackage{mathrsfs}
\def\squarebox#1{\hbox to #1{\hfill\vbox to #1{\vfill}}}

\newcommand{\1}{{\bold 1}}

\theoremstyle{plain}

\newtheorem{thm}{Theorem}
\newtheorem{cor}{Corollary}
\newtheorem{lem}{Lemma}
\newtheorem{rem}{Remark}
\newtheorem{prop}{Proposition}

\newtheorem{exa}{Example}

\def\pu{\partial_t u}
\def\la{\langle}
\def\ra{\rangle}
\def\12{\frac{1}{2}}

\def\phi{\varphi}
\def\epsilon{\varepsilon}
\def\kappa{\varkappa}
\def\ep{\epsilon}

\def\lr#1{\langle{#1}\rangle}
\def\dif{\partial}

\def\al{\alpha}
\def\be{\beta}
\def\gam{\gamma}
\def\R{{\mathbb R}}

\numberwithin{equation}{section}

\begin{document}

\def\R {{\mathbb{R}}}
\def\N {{\mathbb{N}}}
\def\C {{\mathbb{C}}}
\def\Z {{\mathbb{Z}}}
\def\T{{\mathbb T}}
\def\Q{{\mathbb Q}}
\def\SP{{\mathbb S}}
\def\mc{{\mathcal H}}
\def\1b{{\mathbb I}}
\def\pu{\partial_t u}
\def\pt{\partial_t}
\def\pa{\partial}

 \vbadness=10000
 \hbadness=10000

\def\phi{\varphi}
\def\epsilon{\varepsilon}
\def\kappa{\varkappa}
\def\eT{e^{-\lambda T}}
\def\ii{{\bf i}}
\def\hh{\hat{x}}
\def\hhx{\hat{\xi}}
\def\12{\frac{1}{2}}
\def\Rc{{\mathcal R}}
\def\tE{\tilde{E}}
\def\ep{\epsilon}
\def\la{\langle}
\def\ra{\rangle}
\def\co{{\mathcal O}}
\def\pa{\partial}
\def\eN{e^{\frac{N}{t}}}
\def\et{e^{\frac{N}{\tau}}}
\def\op{{\rm Op}}

\def\Qed{\qed\par\medskip\noindent}

\title[Cauchy problem for hyperbolic operators ]{Cauchy problem for hyperbolic operators with triple effective characteristics on the initial plane} 
\date{}
\author[T. Nishitani] {Tatsuo Nishitani}

\address{Departement of Mathematics, Osaka University, Machikaneyama 1-1, Toyonaka 560-0043, Japan}
\email{nishitani@math.sci.osaka-u.ac.jp}
\author[V. Petkov]{Vesselin Petkov}
\address{Institut de Math\'ematiques de Bordeaux, 351,
Cours de la Lib\'eration, 33405  Talence, France}
\email{petkov@math.u-bordeaux.fr}

\begin{abstract} We study effectively hyperbolic operators $P$ with triple characteristics points lying on $t= 0$. Under some conditions on the principal symbol of $P$ one proves that the Cauchy problem for $P$ in $[0, T] \times U$ is well posed for every choice of lower order terms. Our results improves those in \cite{NP} since we don't assume the condition (E) of \cite{NP} satisfied.
\end{abstract}

\subjclass[2010]{Primary 35L30, Secondary 35S10}
\keywords
{Cauchy Problem, Effectively Hyperbolic Operators, Triple
  Characteristics, Energy Estimates}

\vspace{0.5cm}
\maketitle
\section{Introduction}

In this paper we study the Cauchy problem for a differential operator
 $$
 P(t, x, D_t, D_x) = \sum _{k + |\al| \leq 3} c_{k, \al} (t, x) D_t^{k} D_x^{\al},\;\; D_t = -i\partial_t,\;\; D_{x_j} = -i\partial_{x_j}
 $$ 
 of order $3$ with smooth coefficients $c_{k,\alpha}(t,x),\: t \in \R,\:x \in \Omega,\: c_{3, 0}\equiv 1.$ Denote by 
$$
p(t, x, \tau, \xi) = \sum_{k+|\alpha| = 3} c_{k,\alpha} (t, x) \tau^{k} \xi^{\al}=\tau^3+q_1(t,x,\xi)\tau^2+q_2(t,x,\xi)\tau+q_3(t,x,\xi)
$$
the principal symbol of $P$. With a real symbol $\phi\in S^0_{1,0}$ one can write
\begin{equation}
\label{eq:1.1}
\begin{split}
P= (D_t-\op(\phi)\lr{D})^3+ \op(a)\lr{D}(D_t-\op(\phi)\lr{D})^2- \op(b)\lr{D}^2 (D_t-\op(\phi)\lr{D})\\+ \op(c)\lr{D}^3
-\sum_{j=0}^2\op({\tilde b}_{1j})\lr{D}^j(D_t-\op(\phi)\lr{D})^{2-j}
\end{split}
\end{equation}
which is a differential operator in $t$. Here the symbols $a,\,b,\,c\in S^0_{1,0}$ coincide with
\[
q_1\lr{\xi}^{-1}+3\phi,\quad -\big(q_2\lr{\xi}^{-2}+2\phi q_1\lr{\xi}^{-1}+3\phi^2\big),\quad q_3\lr{\xi}^{-3}+\phi q_2\lr{\xi}^{-2}+\phi^2\lr{\xi}^{-1}+\phi^3,
\]
respectively,  ${\tilde b}_{1j}\in S^0_{0,1},\: j=0,1,2$  (see \cite{H2}), and $\lr{D}$ has symbol $\la \xi \ra = (1 + |\xi|^2)^{1/2}.$
Throughout the paper  we work with symbols $s(t,x,\xi)$ which depend smoothly on $t \in [0, T]$   and we use the Weyl quantization
\[
s(t,x,D)u=({\rm Op}^w(s)u)(x)=(2\pi)^{-n}\int\int e^{i(x-y)\xi }s\Bigl(t, \frac{x+y}{2},\xi\Bigr)u(y)dyd\xi.
\]
%
We will use the notation $S^{m}_{0,1}$ for the class of symbols (see \cite{H2}) and we abbreviate  $S^m_{1,0}$ to $S^m$ and ${\rm Op}^w(s)$ to $\op(s)$.
First we assume  that the principal symbol
\begin{equation} \label{eq:1.2}
p(t, x, \tau, \xi) =(\tau-\phi \lr{\xi})^3 + a \lr{\xi}(\tau-\phi \lr{\xi})^2 -b\lr{\xi}^2(\tau-\phi \lr{\xi})  +c\lr{\xi}^3
\end{equation}
is hyperbolic, that is the roots of equation $p = 0$ with respect to $\tau$ are real for $(t, x, \xi) \in [0, T] \times \Omega \times \R^n,$ where $\Omega \subset \R^n$ is an open set. Second we assume also that $p$ has triple characteristic points only if $t = 0$ and $P$ is {\it effectively hyperbolic} (see \cite{IP}, \cite{NP}) at every triple characteristic points $\rho = (0, x, \tau, \xi)$ which is equivalent (see \cite{IP}) to the condition
$$
 \frac{\partial^2 p}{\partial t \partial \tau}(\rho) < 0.
 $$
Recall that an operator is effectively hyperbolic if the fundamental matrix $F_p(z)$ of  the principal symbol $p$ has two non-vanishing eigenvalues $\pm\mu(z)$ at every point $z$ where $dp(z) = 0$. An effectively hyperbolic operator may have triple characteristics only for $t = 0$ or $t = T$ (see \cite{IP}).
Consequently, at a triple characteristic point $\rho_0 = (0, x_0, 0, \xi_0)$, assuming $\phi(0,x_0,\xi_0)=0$,   we have $b_t(0, x_0, \xi_0) >0.$ Moreover, at $\rho_0$ we have $a(0, x_0, \xi_0) = b(0, x_0, \xi_0) = c(0, x_0, \xi_0) =0$.\\

Our purpose is to study the Cauchy problem for $P$ and to prove that under some conditions on $p$ this problem is well posed for every choice of lower order terms. This property is called {\it strong hyperbolicity} and the effective hyperbolicity of $P$ is a necessary condition for it (\cite{IP}). For operators having only double characteristics every effectively hyperbolic operator is strongly hyperbolic and we refer to \cite{N2} for the related works. The conjecture is that effectively hyperbolic operators with triple characteristic points on $t= 0$ are strongly hyperbolic (see \cite{IP},\cite{BBP}, \cite{NP}). On the other hand, for some class of hyperbolic operators with triple characteristics the above conjecture has been proved in \cite{I}, \cite{BBP}, \cite{NP}, but the general case is still an open problem.\\

In \cite{NP} the strong hyperbolicity was established under the condition (E) saying that for some $\delta > 0$ we have the lower bound
$$\frac{\Delta}{\la \xi \ra^6}  \geq \delta t \Bigl(\frac{\Delta_0}{\la \xi \ra^2}\Bigr)^2.$$
Here $\Delta \in S^6$ is the discriminant of the equation $p = 0$ with respect to $\tau$, while
$\Delta_0 \in S^2$ is the discriminant of the equation $\frac{\pa p}{\pa \tau} = 0$ with respect to $\tau.$
In \cite{NP} it was introduced also a weaker condition (H) saying that with some constant $\delta > 0$ we have
$$\frac{\Delta}{\la \xi \ra^6}  \geq \delta t^2 \frac{\Delta_0}{\la \xi \ra^2}.$$
We can consider a microlocal version of the conditions (E) and (H) assuming the above inequalities hold for $(t, x, \xi)$ in a small conic neighborhood $W_0$ of every triple characteristic point $(0, x_0, \xi_0).$
The main goal of this paper  will be stated in Theorem \ref{pro:energy} and Corollary \ref{cor:main} which improve the results in \cite{NP} and  show that we have a strong hyperbolicity for some operators for which (E) is not satisfied, but (H) holds. In particular, we cover the case of operators whose principal symbol $p$ admits a microlocal factorization with one smooth root under the condition that there are no double characteristic points of $p$ converging to a triple characteristic point $(0, x, 0,\xi)$ (see Example 1.1).\\

 Concerning the symbols $a(t, x, \xi),\:b(t, x, \xi),\: c(t, x, \xi),$ we assume the existence of $\delta_1 > 0$ such that
 \begin{equation} 
 \label{eq:miki}
 \begin{split}
&b(t,x,\xi)\geq \delta_1 t, \\
& c=\co(b^2),\;\;\la \xi \ra^{\al}\dif_{\xi}^{\al}\dif_{x}^{\be}c=\co(b),\;|\al+\be|=1,\;\;\la \xi \ra^{\al}\dif_{\xi}^{\al}\dif_{x}^{\be}c=\co(\sqrt{b}),\;|\al+\be|=2,\\
&\dif_tc=\co(b),\;\;\la \xi \ra^{\al}\dif_{\xi}^{\al}\dif_{x}^{\be}(ac)=\co(\sqrt{b}),\;\;|\al+\be|=3.
 \end{split}
 \end{equation}
It is clear that the condition \eqref{eq:miki} are satisfied if 
\begin{equation}
\label{eq:miki:1}
\begin{split}
b(t,x,\xi)\geq \delta_1t,\;\;
 \la \xi \ra^{\al}\dif_t^{\gamma}\dif_{\xi}^{\al}\dif_{x}^{\be}c=\co\big(b^{2-|\al+\be|/2-|\gamma|}\big)\;\;\text{for}\;\; |\al+\be+\gamma|\leq 3,\;\;\gamma=0,1.
\end{split}
\end{equation}
In fact, we assume a slightly weaker microlocal conditions formulated in (\ref{eq:joken}) and Theorem \ref{pro:energy}.\\

Below we present two examples of operators with triple characteristics on $t = 0$ satisfying the above assumptions.
\begin{exa} \rm Assume $c \equiv 0.$ Then the symbol $p$ becomes $p = ((\tau-\phi\lr{\xi})^2 + a \la \xi\ra (\tau-\phi\lr{\xi}) - b \la \xi \ra^2)(\tau-\phi\lr{\xi})$.
 Let $\rho = (0,x_0, \varphi(0, x_0, \xi_0)\la \xi_0\ra,\xi_0),$  be a triple characteristic point. 
For small $t > 0$ we have $b(t, x_0, \xi_0) > 0.$ If for some
$(y, \eta)$ sufficiently close to $(x_0, \xi_0)$ we have $b(0,y, \eta) <0$, then there exists $z =(t^*, x^*, \xi^*)$ with $t^* >0$ such that $b(z) = 0$ and the equation $(\tau-\phi\lr{\xi})^2 + a \lr{\xi}(\tau-\phi\lr{\xi})-b \lr{\xi}^2= 0$ has a root $\phi(z)\lr{\xi^*}$ for $z$. This implies the existence of a double characteristic point $(t^*, x^*, \varphi(z)\la \xi^* \ra, \xi^*)$ of $p$. We exclude this possibility, assuming $b(0, x,\xi) \geq 0$ for $(x, \xi)$ close to $(x_0,\xi_0).$
\end{exa}
\begin{rem}\rm  For the operator in Example $1.1$, the discriminant of the equation $p = 0$ has the form $\Delta = b^2( a^2 + 4b) \la \xi \ra ^6,$ while $\Delta_0 = 4(a^2 + 3 b)\la \xi\ra^2.$ Therefore the condition $(E)$ is reduced to
$$
b^2( a^2 + 4b) \geq \delta t(a^2+ 3b)^2.
$$
If $b=\co(t),$ this inequality yields $b^2a^2+4b^3\geq \delta t a^4$ and hence $a^2\leq \co(t^2)/\delta t=\co(t)$ which is not satisfied in any small neighborhood of a triple characteristic point $(0, x_0, \varphi(0, x_0,\xi_0) \la \xi_0\ra, \xi_0)$, unless $a(0,x,\xi)=0$ for all $(0,x,\xi)$ close to the point $(0,x_0,\xi_0)$. On the other hand, the inequality
$$
b^2(a^2 + 4b) \geq \delta t^2(a^2 + 3b)
$$
obviously holds ($b\geq \delta_1t$ is assumed), hence $(H)$ is satisfied.
\end{rem}
The Example 1.1 covers the case when the principal symbol $p$ admits a factorization
$$
p = (\tau^2 + 2d(t, x,\xi) \tau + f(t, x, \xi))(\tau - \lambda(t, x, \xi))
$$
with $C^{\infty}$ smooth real root $\lambda(t, x, \xi)$ and $p$ has not double characteristic points in a neighborhood of $(0, x_0,\xi_0)$. In fact, we may write
$$
p = \big((\tau- \lambda)^2 + 2(\lambda + d)(\tau- \lambda) + \lambda^2 + 2d \lambda +f\big) (\tau- \lambda)
$$
and taking $\phi=\lambda \lr{\xi}^{-1}$ we reduce the symbol to Example 1.1. Notice that effectively hyperbolic operators with principal symbols admitting above factorization have been studied by V. Ivrii in \cite{I} who proved the strong hyperbolicity  constructing parametrix. Here we present another proof based on energy estimates with weight $t^{-N},$
assuming $P$ strictly hyperbolic for small $t > 0$.
\begin{exa}  \rm Consider the operator with principal symbol 
$$
p = \tau^3 -(t + \alpha(x,\xi))\la \xi \ra^2 \tau  -  (t^2b_2 + tb_1 + b_0)\la \xi \ra ^3,
$$ 
where $\alpha, b_0, b_1, b_2$ are zero order pseudo-differential operators and $\alpha \geq 0.$ This class of operators has been studied in \cite{NP} under the condition $(E)$. We write $p$ as follows
$$
p = (\tau + b_1 \la \xi \ra)^3 -3 b_1 \la \xi \ra(\tau + b_1 \la \xi \ra)^2 - \big(t+\al-3b_1^2 \big)\la \xi \ra^2(\tau+ b_1 \la \xi \ra)
$$
$$ 
- \big[ t^2 b_2+ b_0 -b_1 \alpha+b_1^3\big] \la \xi \ra^3.
$$
Choosing $\phi=-b_1(t,x,\xi)$ one reduces the symbol $p$ to the form $(\ref{eq:1.2})$ with
$a =- 3b_1,\:b = t + \al- 3 b_1^2,\: c = -(t^2 b_2 +b_0 - b_1\al+b_1^3).$ If $\al\geq  3b_1^2,\: b_0 =b_1 \al-b_1^3,$ the condition \eqref{eq:miki:1} is satisfied, while for $\al = 3b_1^2,\: b_0 = b_1 \al - b_1^3$ the condition $(E)$ is not satisfied for $b_1$, unless $b_1(0,x,\xi)\equiv 0$. It is easy to see that with the above choice of $b_0$ and $b_1$, the condition $(H)$ holds.
\end{exa}

Notice that if $\rho = (t, x, \tau, \xi)$ with $t > 0$ is a double characteristic point for $p$, one has $\Delta(\rho) = 0$ and $\Delta_0(\rho) > 0$. Therefore the condition (H) is not satisfied and the analysis of this case is a difficult open problem.
The proofs in this work are based on energy estimates with weight $t^{-N}$ with  $N \gg 1$ leading to estimates with big loss of regularity. This phenomenon is typical for effectively hyperbolic operators with multiple characteristics (see \cite{IP}, \cite{BBP}, \cite{NP}). We follow the approach in \cite{NP} reducing the problem to the one for first order pseudo-differential system. In Section 2 we construct a symmetrizer $S$ for the principal symbol of the system following a general result (see Lemma 2.1) which has independent interest. Moreover, $\det S = \frac{1}{27} \Delta$ and under our assumptions one shows that $\det S \geq \delta b^2(a^2 + 4b),\: \delta > 0.$ Therefore $\Delta \geq \ep t^2(a^2 + 4b),\: \ep > 0,$ and in general the condition (E) is not satisfied. This leads to difficulties in Section 3, where a more fine analysis of the matrix pseudo-differential operators is needed. In Section 4 we show that the microlocal conditions (\ref{eq:miki}) are sufficient for the energy estimates in Theorems 4.1 and 4.2.

\section{Symmetrizer}

First we recall a general result concerning the existence of a symmetrizer. Let $p(\zeta)=\zeta^m+a_1\zeta^{m-1}+\cdots+a_m$ be a monic hyperbolic polynomial of degree $m$ and let $q(\zeta)=p'(\zeta)$. Here $a_j(t, x,\xi)$ depend on $(t, x, \xi)$ but we omit this in the notations below. Let 
\[
h_{p,q}(\zeta,\bar\zeta)=\frac{p(\zeta)q(\bar\zeta)-p(\bar\zeta)q(\zeta)}{\zeta-\bar\zeta}=\sum_{i,j=1}^{m}h_{ij}\zeta^{i-1}{\bar\zeta}^{j-1}
\]
be the B\'ezout form of $p$ and $q$. It is well known that the matrix $H=(h_{ij})$ is nonnegative definite (see for example \cite{I0}).

Consider the Sylvester matrix $A_p$ corresponding to $p(\zeta)$ which has the form
\[
A_p=\left(\begin{array}{ccccc}
0&1&\cdots&0\\
\vdots&\vdots&\ddots\\
0&0&\cdots&1\\
-a_m&-a_{m-1}&\cdots&-a_1
\end{array}\right).
\]
One has the following result \cite{N4} and for the sake of completeness we present the proof.
\begin{lem}{\rm (}\cite[Lemma 2.3.1]{N4}{\rm )}
$H$  symmetrizes $A_p$ and ${\rm det }\,H=\Delta^2$ where $\Delta$ is the
difference-product of the roots  of $p(\tau) = 0$.
\end{lem}
\begin{proof} We first treat the case when $p(\zeta)$ is a strictly hyperbolic polynomial. Let $\lambda_j$, $j=1,\ldots,m$ be the different roots of the equation $p(\zeta)=0$. Write $p(\zeta)=\prod_{j=1}^m(\zeta-\lambda_j)$ and set
\[
\sigma_{\ell,k}=\sum_{1\leq j_1<\cdots<j_{\ell}\leq m, j_p\neq k}\lambda_{j_1}\cdots\lambda_{j_{\ell}}.
\]
 Since $p'(\zeta)=\sum_{k=1}^m\prod_{j=1, j\neq k}^m(\zeta-\lambda_j)=\sum_{i=1}^m(-1)^{m-i}\sigma_{m-i,k}\zeta^{i-1}$ it is easy to see 
 \[
 h_{ij}=\sum_{k=1}^m(-1)^{i+j}\sigma_{m-i,k}\sigma_{j-1,k}.
 \]
 Denote by $R$ the Vandermonde's matrix having the form
 \[
 R=\left(\begin{array}{ccccc}
1&1&\cdots&1\\
\lambda_1&\lambda_2 &\cdots&\lambda_m\\
\vdots&\vdots&\ddots&\vdots\\
\lambda_1^{m-1}&\lambda_2^{m-1}&\cdots&\lambda_m^{m-1}
\end{array}\right).
\]
Since $\lambda_i\neq \lambda_j$, $i\neq j$, the matrix $R$ is invertible and $|{\rm det}\,R|=|\Delta|$. It is clear that
\[
A_pR=R\left(\begin{array}{ccc}
\lambda_1&\\
&\ddots&\\
&&\lambda_m
\end{array}\right).
\]
Denote by $^{co}\!R=(r_{ij})$ the cofactor matrix of $R$ and by $\Delta(
\lambda_1,\ldots,\lambda_k)$ the difference-product of $\lambda_1,\ldots,\lambda_k$.
 It is easily seen that $r_{ij}$ is divisible by $\Delta_i=\Delta(\lambda_1,\ldots,\lambda_{i-1},\lambda_{i+1},\ldots,\lambda_m)$, hence
 \begin{equation}
 \label{eq:2.1}
 r_{ij}=c_{ij}(\lambda_1,\ldots,\lambda_{i-1},\lambda_{i+1},\ldots,\lambda_m)\Delta_i.
 \end{equation}
Since $r_{ij}$ and $\Delta_i$ are alternating polynomials in $(\lambda_1,\ldots,\lambda_{i-1},\lambda_{i+1},\ldots,\lambda_m)$ of degree $m(m-1)/2-j+1$ and  $(m-1)(m-2)/2$ respectively, then $c_{ij}$ is a symmetric polynomial of degree
\[
 m-j=m(m-1)/2-j+1-(m-1)(m-2)/2.
 \]
  Therefore   $c_{ij}$ is a polynomial in fundamental symmetric polynomials of $(\lambda_1,\ldots,\lambda_{i-1},\lambda_{i+1},\ldots,\lambda_m)$.  Noting that $\Delta_i$ is of degree $m-2$ and $r_{ij}$ ($j\neq m$) is of degree $m-1$ respectively with respect to $\lambda_{\ell}$ ($\ell\neq i$),  one concludes that $c_{ij}$ is of degree $1$ with respect to $\lambda_{\ell}$ ($\ell\neq i$) which proves that
\begin{equation}
\label{eq:2.2}
c_{ij}=(-1)^{i+j}\sigma_{m-j,i}.
\end{equation}
Thus denoting $C=(c_{ij})$ we have $^{t}CC=(h_{ij})=H$. In particular, this shows that the symmetric matrix $H$ is nonnegative definite  as it was mentioned above.

Set $D={\rm diag}\,(\Delta_1,\ldots,\Delta_m)$ and note that $D$ is invertible. Moreover it follows from \eqref{eq:2.1} that $C=D^{-1}(^{co}\!R)=({\rm det}R)D^{-1}R^{-1}$ and hence 
\[
CA_pC^{-1}=D^{-1}(R^{-1}A_pR)D.
\]
It is clear that $CA_pC^{-1}$ is a diagonal matrix because both $R^{-1}A_pR$ and $D$ are diagonal matrices. Then $CA_pC^{-1}={^t}C^{-1}\,{^t}\!A_p{^t}C$ yields ${^t}CCA_p={^t}\!A_p{^t}CC$ which proves that $HA_p$ is symmetric. From $C=({\rm det}R)D^{-1}R^{-1}$ it follows that
\[
C={\rm diag}\,\Big(\pm \prod_{k\neq 1}(\lambda_i-\lambda_k),
\pm \prod_{k\neq 2}(\lambda_i-\lambda_k),\ldots, \pm\prod_{k\neq m}(\lambda_i-\lambda_k)\Big)R^{-1}
\]
and hence $|{\rm det}\,C|=|\prod_{j=1}^m\prod_{k\neq j}^m(\lambda_k-\lambda_j)|/|\Delta|=|\Delta|$. Consequently, ${\rm det}\,H=\Delta^2$ and this completes the proof for strictly hyperbolic polynomial $p(\zeta).$

Passing to the general case, introduce the polynomial
$$p_{\ep}(\zeta) = \Bigl(1 + \ep \frac{\pa}{\pa \zeta}\Bigr)^{m-1} p(\zeta), \: \ep \neq 0.
$$
According to \cite{Nuij}, $p_{\ep}(\zeta)$ is strictly hyperbolic and let $H_{\ep}=$ $^{t}C_{\ep} C_{\ep}$ be the symmetrizer for $A_{p_{\ep}}$ constructed above. Obviously, as $\ep \to 0$, we have $A_{p_{\ep}} \rightarrow A_p$ since the coefficients of $p_{\ep}(\zeta)$ go to the ones of $p(\zeta)$. The roots of $p(\zeta)$ depend continuously on the coefficients and this yields
$\lambda_{j, \ep} \rightarrow \lambda_j$, $\lambda_{j, \ep}$ being the roots of $p_{\ep}(\zeta) = 0.$ The equalities (\ref{eq:2.2}) imply $C_{\ep} \rightarrow C$ and passing to the limit $\ep \to 0$, we obtain the result.
\end{proof}  
Note that $H$ is different from the Leray's symmetrizer (\cite{Le}) since if $B$ is the Leray's symmetrizer, then ${\rm det}\,B=\Delta^{2(m-1)}$. 
 Now consider
\[
{\tilde A}_p=\left(\begin{array}{ccccc}
-a_1&-a_2&\cdots&-a_m\\
1&0 &\cdots&0\\
\vdots&\ddots&\ddots\\
0&\cdots&1&0
\end{array}\right).
\]
 \begin{cor} 
 \label{cor:kusatu}Let $J=(\delta_{i,m+1-j}),$ where $\delta_{ij}$ is the Kronecker's delta. Then ${\tilde H}=JH\,{^t}\!J$ symmetrizes ${\tilde A}_p$ and ${\rm det}\,{\tilde H}=\Delta^2$.
 \end{cor}
 \begin{proof}
 Since ${\tilde A}_p=JA_p\,{^t}\!J$ and ${^t}\!JJ=I$ the proof is immediate.
 \end{proof}

With $U={^t}\big((D_t-\op(\phi)\lr{D})^2u, \lr{D}(D_t-\op(\phi)\lr{D})u, \lr{D}^2u\big)$ the equation $Pu=f$ is reduced
\begin{equation}
\label{eq:redE}
D_tU=\op(\varphi) \lr{D}U +(\op(A)\lr{D}+\op(B))U+F,
\end{equation}
where $F={^t}(f,0,0)$ and
\[
 A(t, x, \xi)=\left(\begin{array}{ccc}
- a  & b &  -c\\
1&0&0\\
0&1&0\end{array}\right),\quad B(t, x, \xi)= \left(\begin{array}{ccc}
b_{10}&b_{11}&b_{12}\\
0&b_{21}&0\\
0&0&b_{32}\end{array}\right),
\]
where $b_{ij}\in S^0_{1,0}$.

Introduce
\[
S(t,x,\xi)=\frac{1}{3}\left(\begin{array}{ccc}
3&2a&-b\\[4pt]
2a&2(a^2+b)&-ab-3c\\[4pt]
-b&-ab-3c&b^2-2ac\end{array}\right)
\]
which is a representation matrix (conjugated by $J$ in Corollary \ref{cor:kusatu}) of the  B\'ezout form of $p(\tau)=\tau^3+a\tau^2-b\tau+c$ and $p'(\tau)$ (see for example \cite{I0}, \cite{N3}). Therefore $S$ symmetrizes $A$ so that  
\begin{equation}
\label{eq:taisho}
S(t,x,\xi)A(t,x,\xi)=\frac{1}{3}\left(\begin{array}{ccc}
-a&2b&-3c\\
2b&ab-3c&-2ac\\
-3c&-2ac&bc\end{array}\right).
\end{equation}

Note that when $c=0$ one has
\[
S_0(t,x,\xi)=\frac{1}{3}\left(\begin{array}{ccc}
3&2a&-b\\[4pt]
2a&2(a^2+b)&-ab\\[4pt]
-b&-ab&b^2\end{array}\right)
\]
and hence 
\[
{\rm det}\,S_0(t,x,\xi)=\frac{1}{27}b^2(a^2+4b).
\].

\begin{lem}
\label{lem:setudo} There exist $\bar\epsilon>0$ and $\delta >0$ such that 
\[
{\rm det}\,S\geq \delta b^2(a^2+b)
\]
if $|ac|\leq \bar\ep\, b^2$ and $|c|\leq \bar\ep\, b^{3/2}$.
\end{lem}
\begin{proof}
Note that
\[
{\rm det}\,S={\rm det}\,S_0+\frac{1}{27}\big\{-4a^3c-18abc-27c^2\}.
\]
Since
\begin{align*}
|a^3c|\leq \bar\ep\, b^2a^2,\;\;
|abc|\leq \bar\ep\, b^3,\;\;
|c^2|\leq \bar\ep^2\,b^3
\end{align*}
choosing $\bar\ep=1/50$ for instance, the assertion is clear.
\end{proof}

\begin{lem}
\label{lem:G2}  There exist $\bar\ep>0$ and $\ep_1>0$ such that
\[
S(t,x,\xi)\gg \ep_1 t\left(\begin{array}{ccc}
1&0&0\\
0&1&0\\
0&0&b\end{array}\right)=\ep_1 t J,
\]
 provided  $|ac|\leq \bar\ep\, b^2$ and $|c|\leq \bar\ep\, b^{3/2}$.
\end{lem}
\begin{proof}
  Since 
\[
3S-\ep_1 t J=\left(\begin{array}{ccc}
3-\ep_1 t&2a&-b\\
2a&2a^2+2b-\ep_1 t&-ab-3c\\
-b&-ab-3c&b^2-\ep_1 tb-2ac\end{array}\right),
\]
one obtains
\[
{\rm det}\,(3S-\ep_1 tJ)={\rm det}\,3S+\ep_1 \co \big(b^2(b+a^2)\big).
\]
Indeed
\begin{align*}
(3-\ep_1 t)(2a^2+2b-\ep_1 t)(b^2-\ep_1 tb- 2 ac)=3(2a^2+2b)(b^2- 2 ac)+\ep_1\co \big(tb(b+a^2)\big),\\
b^2(2a^2+2b-\ep_1 t)=b^2(2a^2+2b)+\ep_1\co \big(t b(b+a^2)\big),\\
4a^2(b^2-\ep_1 t b-2ac)=4a^2(b^2-2ac)+\ep_1\co\big(tba^2\big),\\
(3-\ep_1t)(ab+3c)^2=3(ab+3c)^2+\ep_1\co\big(tb^2).
\end{align*}
Noting $b\geq \delta_1 t$, one gets the above representation and we deduce $\det(3S - \ep_1 tJ) \geq 0$ for small $\ep_1.$
In the same way one treats the principal minors of order 2. For example
\[
(3- \ep_1 t)(2a^2 + 2b - \ep_1 t) - 4a^2 = 2a^2 + 6b - \ep_1 t(2a^2 + 2b) + \ep_1^2 t^2 \geq 2(a^2 + b)(1- \ep_1 t) \geq 0,
\]
\begin{align*}
(3- \ep_1 t)(b^2 - \ep_1 t b - 2 ac) - b^2 = 2b^2 - 6 ac - \ep_1 t(b^2 - 2 ac + 3b) 
+ \ep_1^2 t^2b\\
\geq b^2 - 4ac - 3\ep_1 t b+ (b^2- 2 ac)(1 - \ep_1 t) \\
\geq (1-4\bar\ep)b^2-3\ep_1tb+(1-2\bar\ep)(1-\ep_1t)b^2\geq 0,
\end{align*}
\begin{align*}
(2 a^2 + 2b - \ep_1 t)(b^2 - \ep_1 t b - 2 ac) - (ab+ 3c)^2 \geq a^2 b^2 + 2b^3 - 10 abc - 9c^2 - 4 a^3 c \\
 - 3\ep_1 t b^2 + 2 \ep_1 t ac - 2 \ep_1 t ba^2 \\
 \geq (1-4\bar\ep)a^2b^2+(2-10\bar\ep-9\bar\ep^2)b^3- (3\ep_1  + 2 \ep_1 \bar\ep) t b^2 - 2 \ep_1 t ba^2\geq 0
\end{align*}
since all terms involving $\ep_1 t$ can be compensated by $a^2 b^2 + 2 b^3.$
\end{proof}

\begin{lem}
\label{lem:kagi} Assume $\la \xi \ra^{\al}c_{(\be)}^{(\al)} = \co(\sqrt{b})$ for $|\al+ \be| = 2$ and $\la \xi\ra^{\al}(ac)_{(\be)}^{(\al)} = \co(\sqrt{b})$ for $|\al+ \be| = 3.$ There exists $C>0$ such that for $U \in C^{\infty}(\R_t:C_0^{\infty}(\R^n))$ we have
\[
{\mathsf{Re}}({\rm Op}(S)U,U)\geq \ep_1 t\big(\sum_{j=1}^2\|U_j\|^2+({\rm Op}(b)U_3,U_3)\big)-Ct^{-1}\|\lr{D}^{-1}U\|^2.
\]
\end{lem}

\begin{proof} We will follow the argument of \cite[Section 3]{NP} and we use the notation $\pa_{\xi}^{\al}D_{x}^{\be}Q = Q^{(\al)}_{(\beta)}.$ Recall that we have the representation
\begin{equation}
\label{eq:sGar}
\begin{split}
Q_F-{\rm Op}(Q)={\rm Op}\Big(
\sum_{2\leq |\al+\be|\leq 3}\psi_{\al,\be}(\xi)Q^{(\al)}_{(\be)}\Big)+{\rm Op}(R)
\end{split}
\end{equation}
with  $R\in S^{-2}_{1/2,0}$
and real symbols $\psi_{\al,\be}\in S^{(|\al|-|\be|)/2},$ 
where $Q_F$ is the Freidrichs part of $Q$ (see \cite[Appendix]{NP}, \cite{Fr}) and hence $(Q_FU,U)\geq 0$.

Notice that $b$ is real, hence $({\rm Op}(b)U_3, U_3) = {\mathsf{Re}}\, ({\rm Op}(b) U_3, U_3).$ Setting $Q=S-2\ep_1 t J,$  we have 
$$
{\mathsf{Re}}\,({\rm Op}(S)U,U)={\mathsf{Re}}\,({\rm Op}(Q)U,U)+2\ep_1 t\big(\sum_{j=1}^2\|U_j\|^2+({\rm Op}(b)U_3,U_3)\big),
$$
and it is enough to prove
\begin{equation}
\label{eq:katuo}
\big|{\mathsf{Re}}({\rm Op}\Bigl(\sum_{2 \leq |\al + \be| \leq 3}\psi_{\al\be}Q_{(\be)}^{(\al)}\Bigr)U,U)\big|\leq \ep_1 t\big(\sum_{j=1}^2\|U_j\|^2+({\rm Op}(b)U_3,U_3)\big)+C\ep_1^{-1}t^{-1}\|\lr{D}^{-1}U\|^2.
\end{equation}
Indeed if this is true, then we have 
\begin{align*}
{\mathsf{Re}}({\rm Op}(Q)U,U)\geq (Q_FU,U)-\ep_1 t\big(\sum_{j=1}^2\|U_j\|^2+({\rm Op}(b)U_3,U_3)\big)\\-C\ep_1^{-1}t^{-1}\|\lr{D}^{-1}U\|^2
-C\|\lr{D}^{-1}U\|^2\\
\geq -\ep_1 t\big(\sum_{j=1}^2\|U_j\|^2+({\rm Op}(b)U_3,U_3)\big)-C\ep_1^{-1}t^{-1}\|\lr{D}^{-1}U\|^2.
\end{align*}
 Thus we conclude the assertion.
 
To prove \eqref{eq:katuo}, consider ${\mathsf{Re}}({\rm Op}(\psi_{\al\be}Q_{(\be)}^{(\al)})U,U)$ with $|\al+\be|=2$. Setting $g=b^2-\ep tb-2ac,$ one has
\begin{align*}
Q^{(\al)}_{(\be)}
=\left(\begin{array}{ccc}
0&S^{-|\al|}&S^{-|\al|}\\
S^{-|\al|}&S^{-|\al|}&S^{-|\al|}\\
S^{-|\al|}&S^{-|\al|}&g^{(\al)}_{(\be)}\end{array}\right).
\end{align*}
Here and below $S^{m}$ denotes some symbol in the class $S^{m}$.
This yields
\[
\psi_{\al\be}Q^{(\al)}_{(\be)}=\left(\begin{array}{ccc}
0&S^{-1}&S^{-1}\\
S^{-1}&S^{-1}&S^{-1}\\
S^{-1}&S^{-1}&\psi_{\al\be}g^{(\al)}_{(\be)}\end{array}\right)
\]
and hence
\begin{align*}
|({\rm Op}(\psi_{\al\be}Q^{(\al)}_{(\be)})U,U)|\leq \ep_1 t\sum_{j=1}^2\|U_j\|^2+C\ep_1^{-1}t^{-1}\|\lr{D}^{-1}U\|^2\\
+|{\mathsf{Re}}\,\big({\rm Op}\big(\psi_{\al\be}g^{(\al)}_{(\be)}\big)U_3,U_3\big)|.
\end{align*}
Let $
T=\psi_{\al\be}g^{(\al)}_{(\be)}\lr{\xi}.$ 
Then $\psi_{\al\be}g^{(\al)}_{(\be)}={\mathsf{Re}}\,(T\#\lr{\xi}^{-1})+S^{-2}$  and
\[
{\mathsf{Re}}\,({\rm Op}(\psi_{\al\be}g^{(\al)}_{(\be)})U_3,U_3)\leq \ep_1 t\|{\rm Op}(T)U_3\|^2+C\ep_1^{-1}t^{-1}\|\lr{D}^{-1}U_3\|^2.
\]
Note that $\|{\rm Op}(T)U_3\|^2=({\rm Op}(T\#T)U_3,U_3)$ and $T\#T=T^2+S^{-2}$. Therefore there exists $C>0$ such that 
\[
T^2\leq Cb
\]
because $\la \xi \ra^{\al}c^{(\al)}_{(\be)}=\co(\sqrt{b})$ and $\la \xi \ra^{\al}\Bigl(b(b- \ep_1 t)\Bigr)^{(\al)}_{(\be)} = \co(\sqrt{b})$ and $b\geq \delta t$. Applying the Fefferman-Phong inequality for the operator with symbol $Cb - T^2,$ one proves the assertion. 

For the case $|\al+\be|=3$
with $T_1=\psi_{\al\be}g^{(\al)}_{(\be)}\lr{\xi}^{3/2}$ we have  the inequality
\[
T_1^2\leq Cb
\]
with some $C>0$. Indeed, $\la \xi\ra^{\al}(ac)^{(\al)}_{(\be)}=\co(\sqrt{b})$ and $\la \xi \ra^{\al}\Bigl(b(b- \ep_1 t)\Bigr)^{(\al)}_{(\be)} = \co(\sqrt{b})$.  Repeating the above argument, we complete the proof.
\end{proof}

\begin{cor}
\label{cor:seiti:1}Let ${\tilde S}=S+\lambda\, t^{-1}\lr{\xi}^{-2}I$. Then there exists $\lambda_0>0$  such that for $\lambda \geq \lambda_0$ we have
\[
\begin{aligned}
{\mathsf{Re}}(\op({\tilde S})U,U)={\mathsf{Re}}(\op(S)U,U)+\lambda t^{-1}\|\lr{D}^{-1}U\|^2\\
\geq \ep_1 t\big(\sum_{j=1}^2\|U_j\|^2+({\rm Op}\:(b)U_3,U_3)\big)+(\lambda/2) t^{-1}\|\lr{D}^{-1}U\|^2.
\end{aligned}
\]

\end{cor}
\begin{cor}
\label{cor:seiti:2}There exist $\delta_2>0$ and $\lambda_0>0$  such that 
\[
{\mathsf{Re}}(\op({\tilde S})U,U)\geq \delta_2 t^2\|U\|^2+(\lambda/2) t^{-1}\|\lr{D}^{-1}U\|^2,\quad \lambda\geq \lambda_0.
\]
\end{cor}
\begin{proof}Since  there exists $\delta_1>0$ such that $b\geq \delta_1 t$ from the Fefferman-Phong inequality for the scalar symbol $b - \delta_1 t$ one deduces 
\[
({\rm Op}\: (b)U_3,U_3)\geq \delta_1 t\|U_3\|^2-C\|\lr{D}^{-1}U_3\|^2
\]
which proves the assertion thanks to Corollary \ref{cor:seiti:1}.
\end{proof}
%

\section{Energy estimates}

Consider the energy $(t^{-N} e^{-\gam t}\op({\tilde S})U,U)$, where $(\cdot,\cdot)$ is the $L^2(\R^n)$ inner product and $N>0$, $\gam>0$ are positive parameters. Then one has
\begin{equation}
\label{eq:kiso}
\begin{split}
\dif_t(t^{-N}e^{-\gam t}\op({\tilde S})U,U)=-N(t^{-N-1} e^{-\gam t}\op({\tilde S})U,U)-\gam (t^{-N} e^{-\gam t}\op({\tilde S})U,U)\\
+ (t^{-N} e^{-\gam t}\op(\dif_t S)U,U)
-\lambda (N+1)t^{-N-2} e^{-\gam t}\|\lr{D}^{-1}U\|^2 - \lambda \gam t^{-N-1} e^{-\gam t}\|\lr{D}^{-1} U\|^2\\
 -2{\mathsf{Im}}\,(t^{-N} e^{-\gam t}\op({\tilde S})(\varphi \lr{D} +\op( A)\lr{D}+\op(B))U,U))-2{\mathsf{Im}}(t^{-N} e^{-\gam t}\op({\tilde S})F,U).
\end{split}
\end{equation}

Consider $S\#A\#\lr{\xi}-\lr{\xi}\#A^*\#S$. Note that
\[
S\#A=SA+\sum_{|\al+\be|=1}\frac{(-1)^{|\be|}}{2i}S^{(\al)}_{(\be)}A^{(\be)}_{(\al)}+\sum_{|\al+\be|=2}\cdots +S^{-3}.
\]
Writing $S=(s_{ij})$ one has
\[
\sum_{|\al+\be|=2}\cdots=\sum_{|\al+ \be| = 2}\cdots  \Big(s^{(\al)}_{ij(\be)}\Big)\left(\begin{array}{ccc}
-a^{(\be)}_{(\al)}&b^{(\be)}_{(\al)}&-c^{(\be)}_{(\al)}\\
0&0&0\\
0&0&0\end{array}\right)=\left(\begin{array}{ccc}
S^{-2}&S^{-2}&\co(\sqrt{b})S^{-2}\\
S^{-2}&S^{-2}&\co(\sqrt{b})S^{-2}\\
S^{-2}&S^{-2}&\co(\sqrt{b})S^{-2}\end{array}\right),
\]
because $c^{(\be)}_{(\al)}=\co(\sqrt{b})$ for $|\al+\be|=2$. Then 
\[
(S\#A)\#\lr{\xi}=(SA)\#\lr{\xi}+\Big(\sum_{|\al+\be|=1}\cdots\Big)\#\lr{\xi}+\left(\begin{array}{ccc}
S^{-1}&S^{-1}&\co(\sqrt{b})S^{-1}\\
S^{-1}&S^{-1}&\co(\sqrt{b})S^{-1}\\
S^{-1}&S^{-1}&\co(\sqrt{b})S^{-1}\end{array}\right)+S^{-2}.
\]
Denoting the third term on the right-hand side by $K_2$, repeating the same arguments as before,  it is easy to see 
\begin{equation}
\label{eq:soba}
|(({\rm Op}(K_2)+{\rm Op}(S^{-2}))U,U)|\leq C \Bigl(\|\lr{D}^{-1}U\|^2+\sum_{j=1}^2\|U_j\|^2+ ({\rm Op}(b)U_3,U_3)\Bigr).
\end{equation}
Now we turn to the term with $|\al+\be|=1$.  Note
\[
S^{(\al)}_{(\be)}A^{(\be)}_{(\al)}= \Big(s^{(\al)}_{ij(\be)}\Big)\left(\begin{array}{ccc}
-a^{(\be)}_{(\al)}&b^{(\be)}_{(\al)}&-c^{(\be)}_{(\al)}\\
0&0&0\\
0&0&0\end{array}\right)=\left(\begin{array}{ccc}
S^{-1}&S^{-1}&\co(\sqrt{b})S^{-1}\\
S^{-1}&S^{-1}&\co(\sqrt{b})S^{-1}\\
\co(\sqrt{b})S^{-1}&\co(\sqrt{b})S^{-1}&\co(b)S^{-1}\end{array}\right),
\]
since $c^{(\be)}_{(\al)}=\co(\sqrt{b})$ and $b^{(\al)}_{(\be)}=\co(\sqrt{b})$ for $|\al+\be|=1$ and hence
\[
\big(\sum_{|\al+\be|=1}\cdots\big)\#\lr{\xi}=\left(\begin{array}{ccc}
S^{0}&S^{0}&\co(\sqrt{b})S^{0}+S^{-1}\\
S^{0}&S^{0}&\co(\sqrt{b})S^{0}+S^{-1}\\
\co(\sqrt{b})S^{0}+S^{-1}&\co(\sqrt{b})S^{0}+S^{-1}&\co(b)S^{0}+\co(\sqrt{b})S^{-1}+S^{-2}\end{array}\right)=K_1.
\]
The same arguments proves
\[
|({\rm Op}(K_1)U,U)|\leq C\big(\|\lr{D}^{-1}U\|^2+\sum_{j=1}^2\|U_j\|^2+({\rm Op}(b)U_3,U_3)\big).
\]
Consider $A^*\#S$. We have the representation
\[
A^*\#S=A^*S+\sum_{|\al+\be|=1}\frac{(-1)^{|\be|}}{2i}(A^*)^{(\al)}_{(\be)}S^{(\be)}_{(\al)}+\sum_{|\al+\be|=2}\cdots +S^{-3}=A^*S+{\tilde K}.
\]
Repeating similar arguments, one gets
\[
|({\rm Op}(\lr{\xi}\#{\tilde K})U,U)|\leq C\big(\|\lr{D}^{-1}U\|^2+\sum_{j=1}^2\|U_j\|^2+({\rm Op}(b)U_3,U_3)\big).
\]
Since $A^*S=SA$, taking \eqref{eq:taisho} into account, we see
\begin{align*}
(SA)\#\lr{\xi}-\lr{\xi}\#(A^*S)=(SA)\#\lr{\xi}-\lr{\xi}\#(SA)\\
=\left(\begin{array}{ccc}
S^{0}&S^{0}&\co(\sqrt{b})S^{0}+S^{-1}\\
S^{0}&S^{0}&\co(\sqrt{b})S^{0}+S^{-1}\\
\co(\sqrt{b})S^{0}+S^{-1}&\co(\sqrt{b})S^{0}+S^{-1}&\co(b)S^0+\co(\sqrt{b})S^{-1}+S^{-2}\end{array}\right).
\end{align*}
Summarizing the above estimates, we obtain the following
\begin{lem}
\label{lem:setudo:2} Assume $\la \xi \ra^{\al}c_{(\be)}^{(\al)} = \co(\sqrt{b})$ for $|\al + \be| \leq 2$. There is $C>0$ such that
\[
|({\rm Op}(S\#A\#\lr{\xi}-\lr{\xi}\#A^*\#S)U,U)|\leq C\big(\sum_{j=1}^2\|U_j\|^2+({\rm Op}(b)U_3,U_3)+\|\lr{D}^{-1}U\|^2\big).
\]
\end{lem}
Consider $S\#\phi\#\lr{\xi}-\lr{\xi}\#\phi\#S$, where $\phi\in S^0$ is scalar. Recall
\[
S\#\phi=\phi S+\sum_{|\al+\be|=1}\frac{(-1)^{|\be|}}{2i}S^{(\al)}_{(\be)}\phi^{(\be)}_{(\al)}+\sum_{|\al+\be|=2}\cdots +S^{-3}.
\]
For $|\al+\be|=2$ one has
\[
S^{(\al)}_{(\be)}\phi^{(\be)}_{(\al)}=\left(\begin{array}{ccc}
S^{-2}&S^{-2}&S^{-2}\\
S^{-2}&S^{-2}&S^{-2}\\
S^{-2}&S^{-2}&\co(\sqrt{b})S^{-2}\end{array}\right)
\]
and hence
\[
(S\#\phi)\#\lr{\xi}=(\phi S)\#\lr{\xi}+\Big(\sum_{|\al+\be|=1}\cdots\Big)\#\lr{\xi}+\left(\begin{array}{ccc}
S^{-1}&S^{-1}&S^{-1}\\
S^{-1}&S^{-1}&S^{-1}\\
S^{-1}&S^{-1}&\co(\sqrt{b})S^{-1}+S^{-2}\end{array}\right)+S^{-2}.
\]
Denoting the third term on the right-hand side by $K_2,$ we have the same estimate as \eqref{eq:soba}. Similarly one has
\[
\lr{\xi}\#(\phi\#S)=\lr{\xi}\#(\phi S)+\lr{\xi}\#\Big(\sum_{|\al+\be|=1}\cdots\Big)+\left(\begin{array}{ccc}
S^{-1}&S^{-1}&S^{-1}\\
S^{-1}&S^{-1}&S^{-1}\\
S^{-1}&S^{-1}&\co(\sqrt{b})S^{-1}+S^{-2}\end{array}\right)+S^{-2}
\]
Consider the term with $|\al+\be|=1$ and observe that
\[
S^{(\al)}_{(\be)}\phi^{(\be)}_{(\al)}=\left(\begin{array}{ccc}
S^{-1}&S^{-1}&\co(\sqrt{b})S^{-1}\\
S^{-1}&S^{-1}&\co(\sqrt{b})S^{-1}\\
\co(\sqrt{b})S^{-1}&\co(\sqrt{b})S^{-1}&g^{(\al)}_{(\be)}\phi^{(\be)}_{(\al)}\end{array}\right)
\]
with $g=b^2-2ac$. Therefore  
\begin{equation}
\label{eq:rituin}
\lr{\xi}\#(S^{(\al)}_{(\be)}\phi^{(\be)}_{(\al)})=\left(\begin{array}{ccc}
S^{0}&S^{0}&\co(\sqrt{b})S^{0}+S^{-1}\\
S^{0}&S^{0}&\co(\sqrt{b})S^{-1}+S^{-1}\\
\co(\sqrt{b})S^{0}+S^{-1}&\co(\sqrt{b})S^{0}+S^{-1}&\co(b)S^0+\co(\sqrt{b})S^{-1}+S^{-2}\end{array}\right)
\end{equation}
because $c^{(\al)}_{(\be)}=\co(b)$ for $|\al+\be|=1$ and then
\[
|({\rm Op}(\lr{\xi}\#(S^{(\al)}_{(\be)}\phi^{(\be)}_{(\al)}))U,U)|\leq C\big(\sum_{j=1}^2\|U_j\|^2+({\rm Op}(b)U_3,U_3)+\|\lr{D}^{-1}U\|^2\big).
\]
Similar arguments are applied to $|({\rm Op}(\phi^{(\al)}_{(\be)}S^{(\be)}_{(\al)})U,U)|$. Finally, since
\[
\lr{\xi}\#(\phi S)-(\phi S)\#\lr{\xi}=\left(\begin{array}{ccc}
S^{0}&S^{0}&\co(\sqrt{b})S^{0}+S^{-1}\\
S^{0}&S^{0}&\co(\sqrt{b})S^{-1}+S^{-1}\\
\co(\sqrt{b})S^{0}+S^{-1}&\co(\sqrt{b})S^{0}+S^{-1}&\co(b)S^0+\co(\sqrt{b})S^{-1}+S^{-2}\end{array}\right),
\]
 we obtain
\begin{lem}
\label{lem:setudo:3} Assume $\la \xi \ra^{\al}c_{(\be)}^{(\al)} = \co(b)$ for $|\al + \be| = 1$ and $\la \xi \ra^{\al}c_{(\be)}^{(\al)} = \co(\sqrt{b})$ for $|\al + \be| = 2.$ Then there exists $C>0$ such that
\[
|({\rm Op}(S\#\phi\#\lr{\xi}-\lr{\xi}\#\phi\#S)U,U)|\leq C\big(\sum_{j=1}^2\|U_j\|^2+({\rm Op}(b)U_3,U_3)+\|\lr{D}^{-1}U\|^2\big).
\]
\end{lem}
Combining Lemmas \ref{lem:setudo:2}, \ref{lem:setudo:3} and Corollary \ref{cor:seiti:1}, one concludes that for sufficiently large $N_1 > 0$ we have
\begin{equation} 
\label{eq:3.4}
\begin{split}
-N_1 (\op(\tilde{S}) U, U) -2t{\rm Im}\: (\op(S)( \op(\varphi) \lr{D} +\op(A) \lr{D}) U, U)\\
 \leq (-N_1\ep_1 + 2C)t \big(\sum_{j=1}^2\|U_j\|^2+(\op(b)U_3,U_3)\big)
+(-N_1(\lambda/2) t^{-1} + 2C t)\|\lr{D}^{-1}U\|^2 \leq 0
\end{split}
\end{equation}

Now we pass to the analysis of the term involving $\pa_t S$.
\begin{lem}
\label{lem:G2:bis} Assume $\pa_t c = \co(b)$. For $\ep>0$ sufficiently small we  have
\[
S\gg \ep t \dif_t S.
\]
\end{lem}
\begin{proof}
  Since $\pa_t c = \co(b)$, one has
\[
3S-\ep t \dif_tS=\left(\begin{array}{ccc}
3&2a+\ep\co(t)&-b+\ep\co(t)\\
2a+\ep\co(t)&2a^2+2b+\ep\co(t)&-ab-3c+\ep\co(at)+\ep\co(bt)\\
-b+\ep\co(t)&-ab-3c+\ep\co(at)+\ep\co(bt)&b^2-2ac+\ep \co(bt)\end{array}\right).
\]
 It is not difficult to see that
\[
{\rm det}\,(3S-\ep t\,\dif_tS)={\rm det}\,3S+\ep \co \big(b^2(b+a^2)\big)
\]
because $t=\co(b)$.
\end{proof}

\begin{lem}
\label{lem:kagi:2}  Assume $\pa_t c = \co(b)$, $\la \xi \ra^{\al}c_{(\be)}^{(\al)} = \co(\sqrt{b})$ for $|\al+ \be| = 2$ and $\la \xi \ra^{\al}(ac)_{(\be)}^{(\al)} = \co(\sqrt{b})$ for $|\al+ \be| = 3.$  There exist $\ep > 0$ and $C>0$ such that for $U \in C^{\infty}(\R_t:C_0^{\infty}(\R^n))$ we have
\begin{equation} \label{eq:3.5}
{\mathsf{Re}}({\rm Op}(S-\ep t\,\dif_tS)U,U)\geq -\ep t\big(\sum_{j=1}^2\|U_j\|^2+({\rm Op}(b)U_3,U_3)\big)-Ct^{-1}\ep^{-1}\|\lr{D}^{-1}U\|^2.
\end{equation}
\end{lem}
\begin{proof} Denoting $Q=S-2\ep t\,\dif_tS$,
 it suffices to prove
\begin{equation}
\label{eq:katuo:2}
\big|{\mathsf{Re}}({\rm Op}\Bigl(\sum_{2 \leq |\al + \be| \leq 3}\psi_{\al\be}Q_{(\be)}^{(\al)}\Bigr)U,U)\big|\leq \ep t\big(\sum_{j=1}^2\|U_j\|^2+({\rm Op}(b)U_3,U_3)\big)+C\ep^{-1}t^{-1}\|\lr{D}^{-1}U\|^2.
\end{equation}
Consider ${\mathsf{Re}}\big({\rm Op}\big(\psi_{\al\be}Q_{(\be)}^{(\al)}\big)U,U\big)$ with $|\al+\be|=2$. Note that
\[
\psi_{\al\be}Q^{(\al)}_{(\be)}=\left(\begin{array}{ccc}
0&S^{-1}&S^{-1}\\
S^{-1}&S^{-1}&S^{-1}\\
S^{-1}&S^{-1}&\psi_{\al\be}\big(g^{(\al)}_{(\be)}-\ep t\,(\dif_tg)^{(\al)}_{(\be)}\big)\end{array}\right),
\]
where $g=b^2-2ac$. Consequently, one deduce
\begin{align*}
|({\rm Op}(\psi_{\al\be}Q^{(\al)}_{(\be)})U,U)|\leq \ep t\sum_{j=1}^2\|U_j\|^2+C\ep^{-1}t^{-1}\|\lr{D}^{-1}U\|^2\\
+|{\mathsf{Re}}\,({\rm Op}(\psi_{\al\be}(g^{(\al)}_{(\be)}-\ep t(\dif_tg)^{(\al)}_{(\be)}))U_3,U_3)|.
\end{align*}
Setting
\[
T=\psi_{\al\be}\big(g^{(\al)}_{(\be)}-\ep t(\dif_tg)^{(\al)}_{(\be)}\big)\lr{\xi}\in S^{0},
\]
we obtain ${\mathsf{Re}}\,\big(\psi_{\al\be}\big(g^{(\al)}_{(\be)}-\ep t(\dif_tg)^{(\al)}_{(\be)}\big)\big)=T\#\lr{\xi}^{-1}+S^{-2}$.  Therefore
\[
{\mathsf{Re}}\,\big({\rm Op}\big(\psi_{\al\be}(g^{(\al)}_{(\be)}-\ep t(\dif_tg)^{(\al)}_{(\be)}\big)U_3,U_3)\leq \ep t\|{\rm Op}(T)U_3\|^2+C\ep^{-1}t^{-1}\|\lr{D}^{-1}U_3\|^2
\]
Note that $\|{\rm Op}(T)U_3\|^2=({\rm Op}(T\#T)U_3,U_3)$ and $T\#T=T^2+S^{-2}$. There is $C>0$ such that 
\[
T^2\leq Cb
\]
because $t=\co(b)$ and $\la \xi \ra^{\al}c^{(\al)}_{(\be)}=\co(\sqrt{b})$ so that $Cb-T^2\geq 0$. Then applying the Fefferman-Phong inequality, we prove the assertion. Let $|\al+\be|=3$ then 
with $T_1=\big(\psi_{\al\be}\big(g^{(\al)}_{(\be)}-\ep t(\dif_tg)^{(\al)}_{(\be)}\big)\big)\#\lr{\xi}^{3/2}$ 
\[
T_1^2\leq Cb
\]
with some $C>0$ since $t=\co(b)$ and $\la \xi \ra^{\al}(ac)^{(\al)}_{(\be)}=\co(\sqrt{b})$ and the proof is similar.
\end{proof}

From (\ref{eq:3.5}) setting $N_2 = \ep^{-1}$ and dividing by $\ep$, one deduces
\begin{equation*}
{\mathsf{Re}} (\op(-N_2 S + t \pa_t S)U, U) \leq t\big(\sum_{j=1}^2\|U_j\|^2+({\rm Op}(b)U_3,U_3)\big)+Ct^{-1}\ep^{-2}\|\lr{D}^{-1}U\|^2
\end{equation*}
and applying Corollary \ref{cor:seiti:1}, this implies
\begin{equation} 
\label{eq:3.7}
\begin{split}
-(N_2 + N_3){\mathsf{Re}} (\op(\tilde{S})U,U) + t {\mathsf{Re}} (\op(\pa_t S)U, U)\\
\leq (-N_3 \ep_1 + 1)t\big(\sum_{j=1}^2\|U_j\|^2+({\rm Op}(b)U_3,U_3)\big)
+t^{-1}(C \ep^{-2}- N_3 \lambda)\|\lr{D}^{-1}\|^2.
\end{split}
\end{equation}
Fixing $\ep$ and $N_2$, we choose $N_3$ sufficiently large and we arrange the right hand side of the above inequality to be negative.\\

Next we turn to the analysis of $2{\mathsf{Im}}(\op({\tilde S})\op(B)U,U).$
 Recall that $(\op({\tilde S})U,U)\gg 0$ by Corollary \ref{cor:seiti:2}. Consequently,
\begin{equation}
\label{eq:raiu}
\begin{split}
&2|(\op({\tilde S})\op(B)U,U)|\leq N^{-1/2}(t\op({\tilde S})\op(B)U,\op(B)U)+N^{1/2}(t^{-1}\op({\tilde S})U,U)\\
&=N^{-1/2}(t\op(B^*)\op({\tilde S})\op(B)U,U)+N^{1/2}(t^{-1}\op({\tilde S})U,U)\\
&\leq N^{-1/2}(t^{-1}t^2\op(B^*)\op(S)\op(B)U,U)+N^{1/2}(t^{-1}\op(\tilde{S})U,U)+ C_2\lambda N^{-1/2}\|\lr{D}^{-1}U\|^2.
\end{split}
\end{equation}
%
%
%
\begin{lem}
\label{lem:parao}There exists $N_4>0$ depending on $T$ and $B$ such that for $0 \leq t \leq T$ and any $\epsilon>0$ there exists $D_{\epsilon}>0$ such that
\[
{\mathsf{Re}}\big(\op(N_4S-t^2B^*SB)U,U\big)\geq -\epsilon t(\sum_{j=1}^2\|U_j\|^2+(cU_3,U_3))-D_{\epsilon}t^{-1}\|\lr{D}^{-1}U\|^2.
\]
\end{lem}
\begin{proof}Recall
\[
3S-\ep t^2B^*SB=\left(\begin{array}{ccc}
3+\ep \co(t^2)&2a +\ep \co(t^2)&-b+\ep\co(t^2)\\
2a +\ep \co(t^2)&2(a^2 + b) +\ep \co(t^2)&-ab - 3c+\ep \co(t^2)\\
-b+\ep \co(t^2)&-ab -3c+\ep \co(t^2)&b^2 - 2 ac+\ep \co(t^2)\end{array}\right)
\]
which proves $3S-\ep t^2 B^*SB\gg 0$ with some $\ep = \ep(T) >0$. To justify this, notice that the terms $\ep\co(t^2 b), \: \ep\co(t^2 c), \:\ep\co(t^2 a^2), \: \ep\co(t^4 a)$ can be absorbed by $\det S$ because $b\geq \delta_1t$. For example,
$$
\ep t^4|a| \leq \frac{1}{2}\ep (t^5 + t^3a^2) \leq C\ep tb^2(a^2+b).
$$
Choosing $\ep(T)$ small enough, we obtain the result. Then the rest of the proof is just a repetition of the proof of Lemma \ref{lem:kagi:2}.
\end{proof}
According to Lemma \ref{lem:parao} and \eqref{eq:raiu}, one has
\begin{equation}
\label{eq:3.9}
\begin{split}
2|(\op({\tilde S})\op(B)U,U)|\leq 2N_4^{1/2}t^{-1} (\op(\tilde{S})U,U)+\epsilon t(\sum_{j=1}^2\|U_j\|^2+(\op(b)U_3,U_3))\\
-N_4^{1/2} \lambda t^{-2} \|\lr {D}^{-1}U\|^2 + D_{\epsilon}t^{-1}\|\lr{D}^{-1}U\|^2+ C_2\lambda N_4^{-1/2}\|\lr{D}^{-1}U\|^2.
\end{split}
\end{equation}
Combining the estimates (\ref{eq:3.4}), (\ref{eq:3.7}), (\ref{eq:3.9}), it follows that
\[
\begin{aligned}
\dif_t{\mathsf{Re}}(t^{-N} e^{-\gam t}\op({\tilde S})U,U)\leq -2{\mathsf{Im}}(t^{-N} e^{-\gam t}\op({\tilde S})F,U)\\
-(N-N_1-N_2-N_3-2N_4^{1/2})t^{-N - 1} e^{-\gam t}{\mathsf{Re}}(\op(\tilde{S})U,U)\\
+\Bigl[C_{\ep}-\lambda\Bigl(N + 1 + N_4^{1/2}- \lambda C \ep^{-1} \Bigr)\Bigr]t^{-N - 2} e^{-\gam t}\|\lr{D}^{-1}U\|^2\\
+\epsilon t^{-N} e^{-\gam t}\big(\sum_{j=1}^2\|U_j\|^2+(\op(b)U_3,U_3)\big)\\
-(\gam- D_{\ep}-C_1 \lambda - C t\lambda  N_4^{-1/2})t^{-N-1} e^{-\gam t}\|\lr{D}^{-1}U\|^2.
\end{aligned}
\]
Note that
\[
\begin{aligned}
2|(t^{-N} e^{-\gam t}\op({\tilde S})F,U)|\leq 2(t^{-N + 1} e^{-\gam t}\op({\tilde S})F,F)^{1/2}(t^{-N - 1} e^{-\gam t}\op({\tilde S})U,U)^{1/2}\\
\leq (t^{-N + 1} e^{-\gam t}\op({\tilde S})F,F)+(t^{-N-1} e^{-\gam t}\op({\tilde S})U,U).
\end{aligned}
\]
Denote $N^*=N_1+N_2+N_3+ 2N_2^{1/2} + 2$ and we choose $0 <\ep\leq \ep_1$.  We  fix $\ep$ and $\lambda > 2C_{\ep}$. Next we fix  $N_4$ so that
$$ N_4^{1/2} >\lambda C \ep^{-1} + 1.$$
Then the term with $t^{-N - 2} e^{-\gam t}\|\lr{D}^{-1}U\|^2$  is absorbed.
Finally we choose $N > N^*$ and $\gamma$ such that $\gamma-  D_{\ep} -C_1 \lambda- C \lambda N_4^{-1/2}T\geq 0$. Then we have
\begin{equation}
\label{eq:keiyaku}
\dif_t{\mathsf{Re}}(t^{-N} e^{-\gam t}\op({\tilde S})U,U)\leq (t^{-N + 1} e^{-\gam t}\op({\tilde S})F,F)-(N-N^*){\mathsf{Re}}\,(t^{-N - 1} e^{-\gam t}\op(\tilde{S})U,U).
\end{equation}
Integrating \eqref{eq:keiyaku} in $\tau$ from $\epsilon>0$ to $t$ and taking  Corollary \ref{cor:seiti:2} into account, one obtains
\begin{prop}
\label{pro:senken} Assume that 
\begin{equation}
\label{eq:joken}
\begin{split}
&b\geq \delta_1t,\quad |ac|\leq \bar\ep\,b^2,\quad |c|\leq \bar\ep\,b^{3/2},\\
&\la \xi \ra^{\al}c^{(\al)}_{(\be)}=\co\big(b\big)\;\;\text{for}\;\;|\al+\be|=1,\quad \la \xi \ra^{\al}c^{(\al)}_{(\be)}=\co\big(\sqrt{b}\big)\;\;\text{for}\;\;|\al+\be|=2,\\
&\la \xi \ra^{\al}(ac)^{(\al)}_{(\be)}=\co\big(\sqrt{b}\big),\;\;|\al+\be|=3,\quad \dif_tc=\co\big(b\big)
\end{split}
\end{equation}
hold globally where $\bar\ep$ is given in Lemmas $\ref{lem:setudo}$ and $\ref{lem:G2}.$ 
Then there exist $\delta_2>0, \gamma_0 > 0, N \in \N$ and $C>0$ such that for $\gamma \geq \gamma_0$ and $0 < \ep \leq t \leq T$ we have  for any  $U \in C^{\infty}(\R_t: C_0^{\infty}(\R^n))$ 
\[
\begin{aligned}
\delta_2 t^{-N+2}e^{-\gam t}\|U(t)\|^2+\delta_2(N-N^*)\int_{\epsilon}^t\tau^{-N+1}e^{-\gam \tau}\|U(\tau)\|^2 d\tau\\
\leq C\epsilon^{-N-1}e^{-\gam \epsilon}\|U(\epsilon)\|^2+ \int_{\epsilon}^t\tau^{-N+1}e^{-\gam \tau}(\op({\tilde S})F(\tau),F(\tau))d\tau.
\end{aligned}
\]
\end{prop}

\section{Microlocal energy estimates}
First we prove the following
\begin{lem}
\label{lem:kakutyo}
Assume that \eqref{eq:miki} is satisfied in $[0,T]\times {\tilde W}$ where ${\tilde W}$ is a conic neighborhood of $(x_0,\xi_0)$. Then there exist extensions ${\tilde a}(t,x,\xi)\in S^0$, ${\tilde b}(t,x,\xi)\in S^0$ and ${\tilde c}(t,x,\xi)\in S^0$ of  $a$, $b$ and $c$  such that  \eqref{eq:joken} holds globally.
\end{lem}
\begin{proof} 
Assume that \eqref{eq:miki} is satisfied in $[0,T]\times {\tilde W}$. Choose conic neighborhoods $U$, $V$, $W$ of $(x_0,\xi_0)$ such that $U\Subset V\Subset W\Subset {\tilde  W}$. Take $0\leq \chi(x,\xi)\in S^0$, $0\leq {\tilde \chi}(x,\xi)\in S^0$ such that $\chi=1$ on $V$ and $\chi = 0$ outside $W$ and ${\tilde \chi}=0$ on $U$ and ${\tilde \chi}=1$ outside $V$. Choosing $W$ and $T$ small one can assume that $\chi b$ is small as we please in  $[0,T]\times \R^{2n}$ because $b(0,x_0,\xi_0)=0$.  
%
%
We define the extensions of $a$, $b$, $c$ by 
\[
{\tilde a}=\chi a,\quad {\tilde b}=\chi^2b+M{\tilde \chi},\quad {\tilde c}=\chi^3c
\]
where $M>0$ is a positive constant which we will choose below. Note that
\begin{align*}
|{\tilde a}{\tilde c}|=\chi^4|ac|\leq C|a|\chi^4b^2\leq \bar\ep (\chi^2b)^2\leq\bar\ep\,{\tilde b}^2,\\
|{\tilde c}|=\chi^3|c|\leq C\chi^3b^2=Cb^{1/2}(\chi^2b)^{3/2}\leq \bar\ep\, {\tilde b}^{3/2}
\end{align*}
taking $a(0,x_0,\xi_0)=0$, $b(0,x_0,\xi_0)=0$ into account and choosing $W$ small.

If $(x,\xi)\in V$ then ${\tilde b}(t,x,\xi)=b+M{\tilde \chi}\geq \delta_1t$ and if $(x,\xi)$ is outside $V$ then ${\tilde b}(t,x,\xi)=\chi^2b+M \geq \delta_1t$ 
 for $[0,T]\times\R^{2n}$ choosing $M$ so that $M\geq \delta_1T$. Thus we have
 \[
 {\tilde b}(t,x,\xi)\geq \delta_1t\quad (t,x,\xi)\in [0,T]\times\R^{2n}.
 \]
 We turn to estimate derivatives of ${\tilde c}$ and ${\tilde a}{\tilde c}$. For $|\al+\be|=1$ it is clear that
 \[
 \lr{\xi}^{|\al|}\big|{\tilde  c}^{(\al)}_{(\be)}\big|=\lr{\xi}^{|\al|}\big|(\chi^3 c)^{(\al)}_{(\be)}\big|\leq C(\chi^2b^2+\chi^3b)\leq C_1\chi^2b\leq C_1{\tilde b}.
 \]
 Similarly for $|\al+\be|=2$ one sees
 \[
\lr{\xi}^{|\al|}\big|(\chi^3 c)^{(\al)}_{(\be)}\big|\leq C(\chi b^2+\chi^2 b+\chi^3 \sqrt{b})\leq C_1\chi\sqrt{b}= C_1(\chi^2 b)^{1/2}\leq C_1{\tilde b}^{1/2}.
 \]
For $|\al+\be|=3$, taking $\la \xi \ra^{\al}(a c)^{(\al)}_{(\be)}=\co(\sqrt{b})$ into account, one has
\begin{align*}
\lr{\xi}^{|\al|}\big|({\tilde a}{\tilde c})^{(\al)}_{(\be)}\big|=\lr{\xi}^{|\al|}\big|(\chi^4 a c)^{(\al)}_{(\be)}\big|\\
\leq C(\chi b^2+\chi^2 b+\chi^3 \sqrt{b}+\chi^4\sqrt{b})\leq C_1\chi \sqrt{b}\leq C_1{\tilde b}^{1/2}.
\end{align*}
Since $|\dif_t{\tilde c}|=|\chi^3\dif_tc|\leq C\chi^3b\leq C{\tilde b}$ is obvious the proof is complete.
\end{proof}
\begin{rem}\rm In the proof of Lemma \ref{lem:kakutyo} replacing ${\tilde b}$ by $\chi^2b+M{\tilde \chi}+M'\chi_0(\xi)$ where $\chi_0(\xi)\in C_0^{\infty}(\R^n)$ which is $1$ near $\xi=0$ and $M'>0$ is a suitable positive constant it suffices to assume that \eqref{eq:miki} is satisfied in $[0,T]\times {\tilde W}$ for $|\xi|\geq 1$.
\end{rem}
Let  $V\Subset V_1\Subset  \Omega$ and $u \in C^{\infty}(\R_t: C_0^{\infty}(V))$  Let $\{\chi_{\al}\}$ be a finite partition of unity with $\chi_{\al}(x,\xi)\in S^0$ so that
\[
\sum_{\alpha} \chi^2_{\al}(x,\xi)=\chi^2(x),
\]
where $\chi(x)=1$ on $\overline{ V}$ and ${\rm supp}\,\chi \subset V_1$. We can suppose that ${\rm supp}\,\chi_{\al}\subset V_1$. 
We repeat the argument in \cite[Section 4]{NP}, studying a system
$$
D_t U_{\al} = (\op(\varphi)\lr{D} + \op(A)\lr{D} + \op(B))U_{\al} + F_{\al}
$$
with $U_{\al} =$ $ ^t\big((D_t-\op(\phi)\lr{D})^2 \chi_{\al}u, \lr{D}(D_t-\op(\phi)\lr{D})\chi_{\al}u, \lr{D}^2 \chi_{\al}u).$ One extends the coefficients $a$, $b$, $c$ and $\phi$ outside the support of $\chi_{\al}$ and one can assume that  (\ref{eq:joken}) are satisfied globally. Thus we obtain the following

\begin{thm}
\label{pro:energy}Let  $Y\Subset \Omega$. Assume that for every point $(x_0,\xi_0) \in T^*\Omega \setminus \{0\}$ there exist a conic neighborhood $W \subset T^*\Omega \setminus \{0\}$ and $T(x_0, \xi_0) > 0$ such that the estimates \eqref{eq:joken} are satisfied for $0 \leq t \leq T(x_0, \xi_0)$ and $(x, \xi) \in W$. Then there exist $c>0,\: T_0 > 0,\:\gamma_0 > 0,  C>0$ and $N \in \N$ such that for $\gamma \geq \gamma_0,\: 0 < \epsilon < t \leq T_0$ we have for any $U \in C^{\infty}(\R_t: C_0^{\infty}(Y))$
\begin{equation} 
\label{eq:4.1}
\begin{split}
c\,t^{-N+2}e^{-\gam t}\|U(t)\|^2+c\int_{\epsilon}^t\tau^{-N+1}e^{-\gam \tau}\|U(\tau)\|^2d\tau \\
\leq C\epsilon^{-N-1}e^{-\gam \epsilon}\|U(\epsilon)\|^2+C\int_{\epsilon}^t\tau^{-N+1}e^{-\gam \tau}\|f(\tau)\|^2d\tau.
\end{split}
\end{equation}
\end{thm}
\begin{cor}
\label{cor:energy}Let  $Y\Subset \Omega$. Assume that for every point $(x_0,\xi_0) \in T^*\Omega \setminus \{0\}$ there exist a conic neighborhood $W \subset T^*\Omega \setminus \{0\}$ and $T(x_0, \xi_0) > 0$ such that the estimates \eqref{eq:miki} are satisfied for $0 \leq t \leq T(x_0, \xi_0)$ and $(x, \xi) \in W$. Then the same assertion as in Theorem $4.1$ holds.
\end{cor}
The same argument can be applied for the adjoint operator $P^*$. 
With 
$$V={^t}\big((D_t-\op(\phi)\lr{D})^2v, \lr{D}(D_t-\op(\phi)\lr{D})v, \lr{D}^2v\big)$$
 the equation $P^*v=g$ is reduced to
\begin{equation}
\label{eq:redE}
D_tV=\op(\varphi) \lr{D}V +(\op(A)\lr{D}+\op(\tilde{B}))V+G,
\end{equation}
with $G ={^t}(g,0,0).$ Here the principal symbol is the same, while the lower order terms change. To study the Cauchy problem for $P^*$ in $0 < t  < T$ with initial data on $t = T$ one considers
\begin{equation} 
\label{eq:kiso}
\begin{split}
-\dif_t(t^{N}e^{\gam t}\op({\tilde S})V,V)=-N(t^{N-1} e^{\gam t}\op({\tilde S})V,V)-\gam (t^{N} e^{\gam t}\op({\tilde S})V,V)\\
- (t^{N} e^{\gam t}\op(\dif_t S)V,V)
- \lambda (N- 1)t^{N-2} e^{\gam t}\|\lr{D}^{-1}U\|^2 - \lambda \gam t^{N - 1} e^{\gam t}\|\lr{D}^{-1} U\|^2\\
 +2{\mathsf{Im}}\,(t^{N} e^{\gam t}(\op({\tilde S})(\op(\varphi) \lr{D} + \op(A)\lr{D}+\op({\tilde B}))V,V))+ 2{\mathsf{Im}}(t^{N} e^{\gam t}\op({\tilde S})G,V).
\end{split}
\end{equation}
Repeating the argument of Section 3, one obtains the following

\begin{thm}
\label{pro:energy:bis}Let  $Y\Subset \Omega$. Assume that for every point $(x_0,\xi_0) \in T^*\Omega \setminus \{0\}$ there exist a conic neighborhood $W \subset T^*\Omega \setminus \{0\}$ and $T(x_0, \xi_0) > 0$ such that the estimates \eqref{eq:joken} are satisfied for $0 \leq t \leq T(x_0, \xi_0)$ and $(x, \xi) \in W$. Then there exist $c>0,\: T_0 > 0,\:\gamma_0 > 0,  C>0$ and $N \in \N$ such that for $\gamma \geq \gamma_0,\: 0 < \epsilon < t \leq T_0$ we have for any $V \in C^{\infty}(\R_t: C_0^{\infty}(Y))$
\begin{equation} 
\label{eq:4.4}
\begin{split}
c\,t^{N+2}e^{\gam t}\|V(t)\|^2+c\int_{t}^{T_0}\tau^{N+1}e^{\gam \tau}\|V(\tau)\|^2d\tau \\
\leq C T_0^{N-1}e^{\gam T_0}\|V(T_0)\|^2+C\int_{t}^{T_0}\tau^{N+1}e^{\gam \tau}\|g(\tau)\|^2d\tau.
\end{split}
\end{equation}
\end{thm}

Following the argument in \cite{NP}, we may absorb the weight $\tau^{-N}$ and obtain energy estimates with a loss of derivatives. For the sake of completeness we recall this argument. Consider $Pu=f$ for $u \in C^{\infty}(\R_t: C_0^{\infty}(\R^n))$. Assume $u(\ep, x) = u_t(\ep, x) = u_{tt}(\ep, x) = 0$.
Differentiating $Pu = f$ with respect to $t$, we determine the functions
$D _t^j u(\ep, x) = u_j(x) \in C_0^{\infty}(\R^n)$ and set
$$
u_M(t,x) = \sum_{j=0}^M \frac{1}{j!}u_j(x)(i (t- \ep))^j, \:0 <  \ep \leq t \leq T_0.
$$
Therefore $w = u - u_M \in C^{\infty}(\R_t: C_0^{\infty}(\R^n))$ satisfies $Pw = f_M$ with
$$
D_t^j f_M(\ep, x) = 0,\: j = 0,1,\ldots, M-3,\quad D_t^j w(\ep, x) = 0,\:j = 0,1,\ldots,M.
$$
Consequently, from Theorem \ref{pro:energy} one deduce the existence of $N \in \N$ and $C > 0$ such that for $\ep > 0$, $ Z \Subset \Omega$ and a solution of $Pu = f\in C_0^{\infty}([0, T_0] \times Z)$ for $0 < \ep \leq t \leq T_0$, $x \in Z$ with 
$$
u(\ep, x) = u_t(\ep, x) = u_{tt}(\ep, x) = 0
$$
 we have
\begin{equation} \label{eq:4.5}
\sum_{j +| \al|\leq 2}\int_{\ep}^t \|\pa_t^{j} \pa_x^{\al}u(s, x)\|_{L^2(Z)} ds\leq C \int_{\ep}^t \|\sum_{j + |\al| \leq N} \pa_t^j \pa_x^{\al}Pu(s, x)\|_{L^2(Z)} ds,
\end{equation}
where $C$ is independent on $\ep$. We can obtain a similar estimates for higher order derivatives. \\

By applying the estimate (\ref{eq:4.5}) and the fact that under the assumptions of Theorem \ref{pro:energy} the symbol $p$ is strictly hyperbolic for $0 < t \leq T_0$ one can obtain the existence of a solution of the Cauchy problem in $[0, T_0] \times Z$ repeating the argument in \cite[Theorem 25.4.5]{ H2}. 
The fact that $p$ is strictly hyperbolic for $0 < t \leq T_0$, is equivalent to  $\Delta > 0$ for $0 < t \leq T_0,$ $\Delta$ being the discriminant of the equation  $p = 0$ with respect to $\tau$. On the other hand, $\Delta = 27 \det S$ (see also Corollary \ref{cor:kusatu}) and $\det S > 0$ for $t >0$ by Lemma \ref{lem:setudo}. The local uniqueness of the solution of the Cauchy problem for $P$ can be obtained taking into account Theorem \ref{pro:energy:bis} for the adjoint operator $P^*$ and using the argument of \cite[Theorem 25.4.5]{H2}. We leave the details to the reader.\\

Finally, we deduce

\begin{cor}
\label{cor:main} Under the assumptions of Theorem \ref{pro:energy}  the Cauchy problem for $P$ is $C^{\infty}$ well posed  in $[0, T_0] \times Z$ for all lower order terms.
\end{cor}

\end{document}